\documentclass[reqno,12pt]{amsart}

\usepackage{graphics}
\usepackage{amssymb}
\usepackage{amsmath}
\usepackage{enumitem}

\date{}

\theoremstyle{plain}
\newtheorem{theorem}{Theorem}

\newtheorem{lemma}{Lemma}
\newtheorem{proposition}{Proposition}

\theoremstyle{definition}
\newtheorem{definition}{Definition}
\newtheorem{remark}{Remark}
\newtheorem{remarks}{Remarks}

\def\N{\mathbb N}
\def\Z{\mathbb Z}
\def\R{\mathbb R}

\def\phi{\varphi}
\def\geq{\geqslant}
\def\leq{\leqslant}
\def\tt{t\^ete-\`a-t\^ete}
\def\ul{\underline}

\title{Cutting arcs for torus links and trees}

\author{Filip Misev}
\address{Universit\"at Bern, Sidlerstrasse 5, CH-3012 Bern, Switzerland}
\email{filip.misev@math.unibe.ch}

\begin{document}

\begin{abstract} Among all torus links, we characterise those arising as links of simple plane curve singularities by the property that their fibre surfaces admit only a finite number of cutting arcs that preserve fibredness. The same property allows a characterisation of Coxeter-Dynkin trees (i.e., $A_n$, $D_n$, $E_6$, $E_7$ and $E_8$) among all positive tree-like Hopf plumbings.
\end{abstract}

\maketitle
\thispagestyle{empty}

\section{Introduction}

A {\em fibred link} is a link $L\subset S^3$ such that $S^3\setminus L$ fibers over the circle, and where each fibre is the interior of a Seifert surface $S$ for $L$ in $S^3$. Cutting $S$ along a properly embedded interval $\alpha$ (an {\em arc} for short) results in another Seifert surface $S'$ for another link $\partial S'=L'$. If $L'$ is again a fibred link with fibre $S'$, we say that $\alpha$ {\em preserves fibredness}. For example, $\alpha$ could be the spanning arc of a plumbed Hopf band, and cutting along $\alpha$ amounts to deplumbing that Hopf band. In \cite{BIRS}, Buck et al.\ give a simple criterion for when an arc preserves fibredness in terms of the monodromy $\phi\colon S\to S$. As a corollary, they prove that each of the torus links of type $T(2,n)$ admits only a finite number of such arcs up to isotopy. It turns out that among torus links, this is an exception:

\begin{theorem} \label{thm:1} Let $n,m\geq 4$ or $n=3, m\geq 6$. Then the fibre surface $S$ of the torus link $T(n,m)$ contains infinitely many homologically distinct cutting arcs preserving fibredness.
\end{theorem}

The remaining torus links $T(2,n)$, $T(3,3)$, $T(3,4)$ and $T(3,5)$ happen to be exactly those torus links that can also be obtained as plumbings of positive Hopf bands according to a finite tree, where vertices correspond to positive Hopf bands and edges indicate plumbing.

\begin{theorem} \label{thm:2} Let $S$ be the fibre surface obtained by plumbing positive Hopf bands according to a finite tree $T$. There are, up to isotopy, only finitely many cutting arcs in $S$ preserving fibredness, if and only if $T$ is one of the Coxeter-Dynkin trees $A_n$, $D_n$, $E_6$, $E_7$ or $E_8$.
\end{theorem}

To prove the ``only if'' part of Theorem~\ref{thm:2}, we consider orbits of a fixed arc under the monodromy to produce families of arcs that preserve fibredness. The basic idea is that such an orbit is infinite if the monodromy has infinite order. For example, we show that in fact every (prime) positive braid link with pseudo-Anosov monodromy admits infinitely many non-isotopic arcs preserving fibredness. This suggests the following question: is it true that among all (non-split prime) positive braid links, the ADE plane curve singularities are exactly those that admit just a finite number of fibredness preserving arcs up to isotopy?

\subsection*{Plan of the article}
We use the shorthand {\em $ADE$ links} to refer to the links of the positive tree-like Hopf plumbings according to the trees $A_n$, $D_n$, $E_6$, $E_7$ or $E_8$. The subsequent section combines a criterion on arcs to preserve fibredness from~\cite{BIRS} with the property of monodromies of positive Hopf plumbed surfaces to be right-veering. This allows for the following simple test for an arc to preserve fibredness, in our situation: an arc preserves fibredness if and only if it does not intersect its image under the monodromy (up to free isotopy).

Section~\ref{sec:monodromies} contains descriptions of the fibre surfaces and the monodromies of the links we consider (torus links and the $ADE$ links). Alongside, we give a constructive proof of Theorem~\ref{thm:1}.

In Section~\ref{sec:exceptional}, we explain the idea of proof for the finiteness result that provides the ``if'' part of Theorem~\ref{thm:2}, and list the fibred links obtained by cutting the fibre surfaces of the $ADE$ links along an arc in Table~\ref{table}.

Section~\ref{sec:infinite} accounts for the cases where the monodromy has infinite order. This concerns in particular the positive tree-like Hopf plumbings that correspond to trees different from the $ADE$ trees and settles the ``only if'' part of Theorem~\ref{thm:2}.

At the beginning of Section~\ref{sec:propproofs}, we set up the notation and methods needed for the proof of the finiteness part of Theorem~\ref{thm:2}, which we split into Proposition~\ref{prop:3345} (concerning torus links) and Proposition~\ref{prop:DnE7} (concerning tree-like Hopf plumbings). The rest of that section is devoted to the proofs of these propositions.

\section{Right-veering surface diffeomorphisms and cutting arcs that preserve fibredness} \label{sec:right-veering}

In the sequel we would like to make statements on the relative position of two arcs $\alpha,\beta$ in a surface $S$ with boundary (that is, $\alpha,\beta$ are embedded intervals with endpoints on the boundary of $S$ that are nowhere tangent to $\partial S$). The following definition will simplify matters.

\begin{definition} Let $S$ be an oriented surface with boundary and let $\alpha,\beta\subset S$ be two arcs. A property $P(\alpha,\beta)$ is said to hold {\em after minimising isotopies on $\alpha$ and $\beta$}, if $P(\tilde{\alpha},\tilde{\beta})$ holds, where $\tilde{\alpha}$ and $\tilde{\beta}$ are obtained from $\alpha$, $\beta$ by two isotopies (fixed at the endpoints) that minimise the geometric number of intersections between the two arcs.
\end{definition}

The remainder of this section will recall the fact that every positive braid link (that is, the closure of a braid word consisting only of the positive generators of the braid group, without their inverses) is fibred and has so-called {\em right-veering} monodromy (see below for a definition). The torus links $T(n,m)$ provide examples, since they can be viewed as the closures of the positive braids $(\sigma_1\cdots\sigma_{n-1})^m$, where the $\sigma_i$ denote the (positive) standard generators of the braid group.

\begin{definition}[see \cite{HKM}, Definition~2.1] Let $S$ be an oriented surface with boundary and $\phi : S \to S$ a diffeomorphism that restricts to the identity on $\partial S$. Then, $\phi$ is called {\em right-veering} if for every arc $\alpha : [0,1]\to S$, the vectors $(\alpha'(0), (\phi\circ\alpha)'(0))$ form an oriented basis after minimising isotopies on $\alpha$ and $\phi\circ\alpha$. This means basically that arcs starting at a boundary point of $S$ get mapped ``to the right'' by $\phi$.
\end{definition}

It is known that every positive braid can be obtained as an iterated plumbing of positive Hopf bands (see \cite{St}). Since a Hopf band is a fibre and plumbing preserves fibredness, every positive braid link is fibred. Moreover the monodromy is a product of positive Dehn twists, since the monodromy of a (positive) Hopf band is a (positive) Dehn twist and the monodromy of a plumbing is the composition of the monodromies of the plumbed surfaces (see \cite{Ga}). A product of positive Dehn twists is right-veering \cite{HKM}. So we conclude that every positive braid link is fibred with right-veering monodromy. Together with a theorem by Buck et al., this property implies the following simple geometric criterion for when an arc preserves fibredness.

\begin{figure}[h]  
\includegraphics{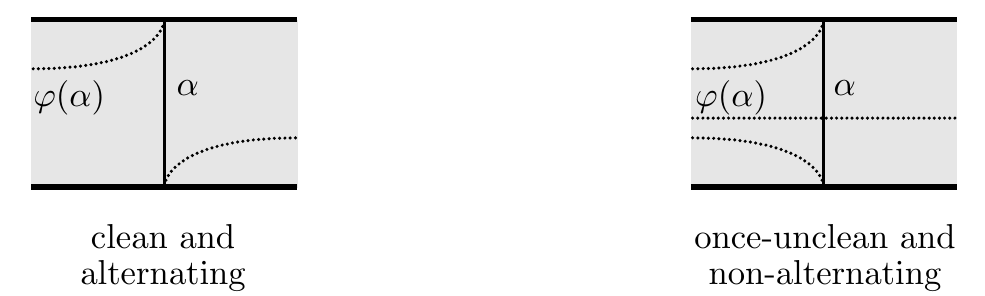}
\caption{(Adapted from \cite{BIRS}) \label{fig:alternating}}
\end{figure}

\begin{theorem}[compare Theorem~1 in \cite{BIRS}] \label{thm:rv-cutting-arcs} Let $L$ be a fibred link with fibre surface $S$ and right-veering monodromy $\phi:S\to S$. Then, a cutting arc $\alpha$ preserves fibredness if and only if $\alpha\cap\phi(\alpha)=\partial\alpha$ after minimising isotopies on $\alpha$ and $\phi(\alpha)$.
\end{theorem}

\begin{proof} This is a special case of Theorem~1 in~\cite{BIRS}, saying that the arc $\alpha$ preserves fibredness if and only if $\alpha$ is {\em clean and alternating} or {\em once unclean and non-alternating} (see Figure~\ref{fig:alternating}), without the assumption on $\phi$ to be right-veering. But for a right-veering map, every arc is alternating, by definition. Finally, $\alpha$ is clean if and only if $\alpha\cap\phi(\alpha)=\partial\alpha$ after minimising isotopies on $\alpha$ and $\phi(\alpha)$.
\end{proof}

\begin{remark} \label{rem:clean}
An arc $\alpha$ is clean if and only if $\phi^k(\alpha)$ is clean, for all $k\in\Z$. This is clear since $\alpha\cap\phi(\alpha)=\partial\alpha$ after minimising isotopies if and only if $\phi^k(\alpha)\cap\phi^{k+1}(\alpha)=\partial\alpha$ after minimising isotopies. Similarly, if $\tau:S\to S$ is a homeomorphism such that $\phi\circ\tau\circ\phi=\tau$, then $\alpha$ is a clean arc if and only if $\alpha'=\tau(\phi(\alpha))$ is. Indeed, $\alpha\cap\phi(\alpha)=\partial\alpha\ \Leftrightarrow \ \tau(\alpha)\cap \tau(\phi(\alpha))=\partial\alpha'\ \Leftrightarrow \ \phi(\tau(\phi(\alpha)))\cap \tau(\phi(\alpha))=\partial\alpha'\ \Leftrightarrow \ \phi(\alpha')\cap \alpha'=\partial\alpha'$.
\end{remark}

\section{Monodromy of torus links, $E_7$ and $D_n$.} \label{sec:monodromies}
The links that correspond to the trees $A_n$, $E_6$ and $E_8$ are torus links, namely $A_{n-1}=T(2,n)$, $E_6=T(3,4)$ and $E_8=T(3,5)$. Together with $D_4$, which is $T(3,3)$, these form the intersection between torus links and positive tree-like Hopf plumbings. For our purpose it therefore suffices to study torus links, $E_7$ and the $D_n$ family.

The monodromies $\phi:S\to S$ of the links in question are particular examples of {\em {\tt} twists}, a notion invented by A'Campo and further developed by Graf in his thesis~\cite{Gr}. This means that there exists a $\phi$-invariant spine $\Gamma\subset S$, called {\em {\tt} graph}. Cutting $S$ along the {\tt} graph results in finitely many annuli, on which $\phi$ descends to certain twist maps. More precisely, each of these annuli has one component of $\partial S$ as one boundary circle and a cycle consisting of edges of $\Gamma$ as the other. $\phi$ fixes $\partial S$ pointwise and rotates the edge-cycles by some number $\ell$ of edges. The number $\ell\in\Z$ is called the {\em twist length} of the corresponding boundary annulus. After an isotopy (fixing the boundary of $S$), we may therefore assume that $\phi$ is periodic except on some annular neighbourhoods of $\partial S$. It is thus easy to understand the effect of $\phi$ on an arc $\alpha$, up to isotopy, given the combinatorics of the action of $\phi$ on $\Gamma$ and the amount of twisting on each annulus. Note that {\tt} twists define periodic mapping classes in the sense that some power is freely isotopic to the identity. However, this isotopy cannot be taken to be fixed on the boundary of $S$.

In a way dual to the {\tt} graph, we will find in each case a finite set of disjoint arcs that are permuted by $\phi$ and which decompose $S$ into finitely many polygons, one for each vertex of $\Gamma$. The combinatorics of how these polygons are permuted will be used to prove Theorems~\ref{thm:1} and~\ref{thm:2}.

\subsection*{Monodromy of torus links}
The fibre surface $S$ of the torus link $T(n,m)$ can be constructed as thickening of a complete bipartite graph on $n$ and $m$ vertices in the following way, as in Figure~\ref{fig:Knm}.
\begin{figure}[h]
\includegraphics{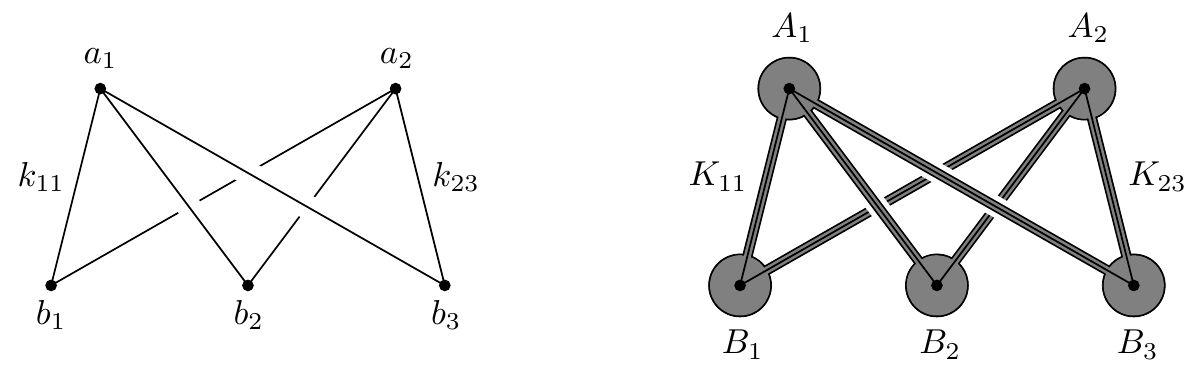}
\caption{The complete bipartite graph on $2$ and $3$ vertices and blackboard framed thickening.\label{fig:Knm}}
\end{figure}
Arrange $n$ collinear points $a_1,\ldots,a_n$ (in this order) in a plane and, similarly, another $m$ points $b_1,\ldots,b_m$ along a line parallel to the $a_i$. Connect $a_i$ and $b_j$ by a straight segment $k_{ij}$, for every $i\in\{1,\ldots,n\}$, $j\in\{1,\ldots,m\}$. Avoid intersections between the segments by letting $k_{ij}$ pass slightly under $k_{pq}$ if $i>p$ and $j<q$ (in a slight thickening of the plane containing the points $a_i$ and $b_j$). Use the blackboard-framing to thicken $a_i$, $b_j$, $k_{ij}$ to disks $A_i$, $B_j$ and bands $K_{ij}$. Choose the thickness of the bands $K_{ij}$ so that they do not intersect outside the disks $A_i, B_j$. It can be seen that $S:=\bigcup_i A_i \cup \bigcup_j B_j \cup \bigcup_{i,j} K_{ij} \subset\R^3\subset S^3$ is isotopic to the minimal Seifert surface of $T(n,m)$ in $S^3$ (compare \cite{Ba}). In addition, the monodromy $\phi:S\to S$ is a {\tt} twist along the above graph. In each of the $\gcd(n,m)$ complementary annuli, $\phi$ fixes $\partial S$ pointwise and rotates the edge-cycles two edges to the right with respect to the orientation of $S$. Using this description, it is possible to see that $\phi$ acts on the graph as follows: $\phi(a_i)=a_{i-1}$, $\phi(b_j)=b_{j+1}$, $\phi(k_{ij})=k_{i-1,j+1}$, where the indices $i,j$ are to be taken modulo $n,m$ respectively. A subarc of $\alpha$ that travels near $k_{ij}$ will be mapped to a subarc of $\phi(\alpha)$ that travels near $k_{i-1,j+1}$. The edges $k_{ij}$ induce a decomposition of $\partial A_i$ into circular arcs lying between points of the form $k_{ij}\cap\partial A_i$ (and the same for $B_j$). If $n,m\geq 3$, it is hence meaningful to speak of points on $\partial A_i$ {\em between} $k_{ij}$ and $k_{i,j+1}$.

\setcounter{theorem}{0}
\begin{theorem} Let $n,m\geq 4$ or $n=3, m\geq 6$. Then the fibre surface $S$ of the torus link $T(n,m)$ contains infinitely many homologically distinct cutting arcs preserving fibredness.
\end{theorem}

\begin{figure}[h]
\includegraphics{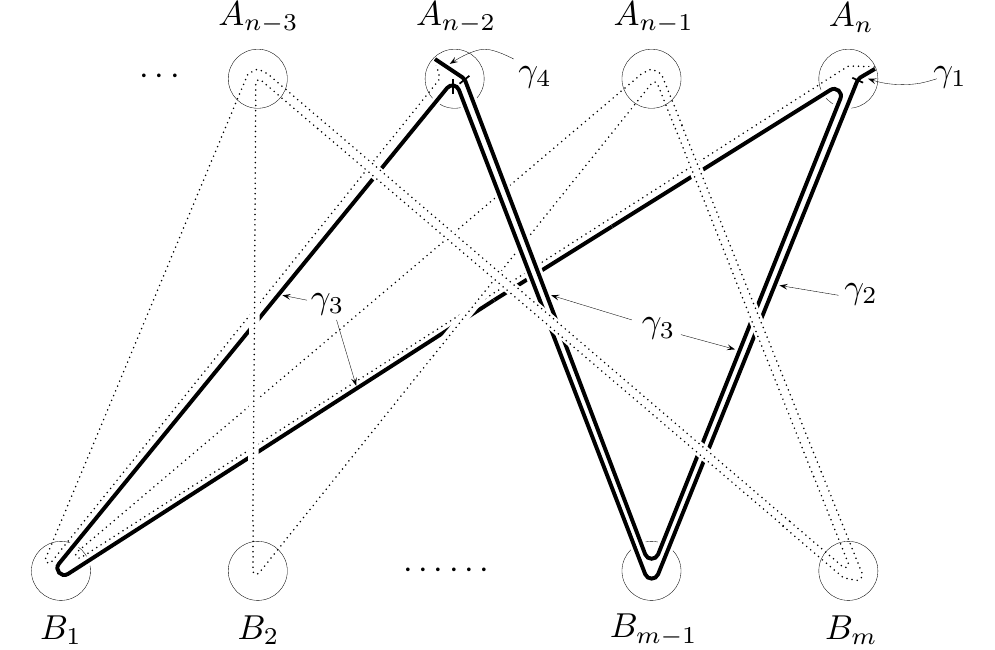}
\caption{The arc $\alpha_1=\gamma_1*\gamma_2*\gamma_3*\gamma_4$ (solid line) and its image under the monodromy (dotted line). Note that these two arcs do not intersect, except at their endpoints.\label{fig:InfiniteIntervalnm}}
\end{figure}

\begin{proof} For $n,m\geq 4$ consider the following arcs in $S$, using the notation from above (compare Figure~\ref{fig:InfiniteIntervalnm}):

\begin{itemize}
\item Let $\gamma_1$ be a straight segment starting at a point of $\partial A_n$ between $k_{n1}$ and $k_{nm}$ and ending at the vertex $a_n$.
\item Let $\gamma_2$ start at $a_n$, follow the edges $k_{n,m-1}$ and $k_{n-2,m-1}$, thus ending at $a_{n-2}$.
\item $\gamma_3$ starts at $a_{n-2}$, runs along $k_{n-2,1}$, $k_{n1}$, $k_{n,m-1}$, $k_{n-2,m-1}$ and ends again at $a_{n-2}$.
\item $\gamma_4$ is a straight segment from $a_{n-2}$ to a point of $\partial A_{n-2}$ between $k_{n-2,1}$ and $k_{n-2,m}$.
\end{itemize}

From $\gamma_1,\gamma_2,\gamma_3,\gamma_4$ we can build an infinite family $(\alpha_r)_{r\in\N}$ of arcs in $S$, taking $\alpha_r = \gamma_1*\gamma_2*\underbrace{\gamma_3*\ldots *\gamma_3}_{r-\text{times}} *\gamma_4$. Here, $*$ denotes concatenation of paths. Replacing the $r$ consecutive copies of $\gamma_3$ by $r$ parallel copies, the $\alpha_r$ can be thought of as embedded arcs. It is now easy to check that $\alpha_r$ and its image $\phi_*\alpha_r$ under the monodromy $\phi$ have only their endpoints in common. Using Theorem~\ref{thm:rv-cutting-arcs} it follows that each $\alpha_r$ preserves fibredness. Finally, the $\alpha_r$ are homologically pairwise distinct. This can be seen in the following way: Let $[c]\in H_1(S,\Z)$ be the cycle represented by a simple closed curve $c$ whose image is $k_{nm}\cup k_{n-1,m}\cup k_{n-1,m-1}\cup k_{n,m-1}$. After an isotopy, $c$ will intersect $\alpha_r$ transversely in $r+1$ points. Now, the linear form on $H_1(S,\partial S,\Z)$ that sends $\alpha$ to $i(c,\alpha)$, the number of intersections with $c$ (counted with signs), defines an element $c^*$ of $H^1(S,\partial S,\Z)$ such that $c^*(\alpha_r)=r+1$, hence the claim.

If $n=3, m\geq 6$, take the following arcs (compare Figure~\ref{fig:InfiniteInterval3m}):
\begin{figure}[h]
\includegraphics{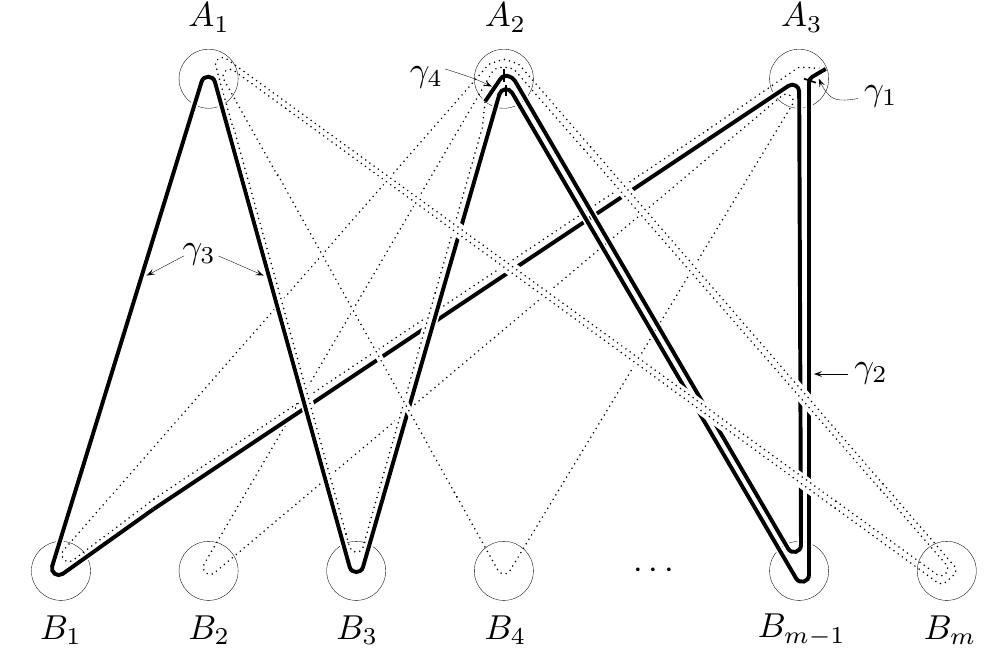}
\caption{The arc $\alpha_1$ (solid line) and its image under the monodromy (dotted line) for a $T(3,m)$ torus link, $m\geq 6$. Again, the two arcs do not intersect.\label{fig:InfiniteInterval3m}}
\end{figure}

\begin{itemize}
\item $\gamma_1$ is a straight segment from a point of $\partial A_3$ between $k_{31}$ and $k_{3m}$ to $a_3$.
\item $\gamma_2$ starts at $a_3$, follows the edges $k_{3,m-1}$ and $k_{2,m-1}$, thus ending at $a_2$.
\item $\gamma_3$ starts at $a_2$, follows $k_{23}$, $k_{13}$, $k_{11}$, $k_{31}$, $k_{3,m-1}$ and $k_{2,m-1}$, ending at $a_2$.
\item $\gamma_4$ is a straight segment from $a_2$ to a point of $\partial A_2$ between $k_{22}$ and $k_{23}$.
\end{itemize}

As above, we get a family $(\alpha_r)_{r\in\N}$ of homologically distinct arcs preserving fibredness, where $\alpha_r = \gamma_1*\gamma_2*\underbrace{\gamma_3*\ldots *\gamma_3}_{r-\text{times}} *\gamma_4$, using the curve with image $k_{3m}\cup k_{1m}\cup k_{1,m-2}\cup k_{3,m-2}$ to distinguish the $\alpha_r$.
\end{proof}

\subsection*{Monodromy of $E_7$ and $D_n$} In order to obtain a similar model for the fibre surface $S$ of $E_7$ or $D_n$, start with two disjoint planar disks $D, D'$ in $\R^3$ and connect them by half twisted bands $b_1,\ldots,b_n$, where $n=7$ in the case of $E_7$. The embedded surface $S'=D\cup D'\cup b_1\cup\ldots\cup b_n$ is then a fibre surface for the $T(2,n)$ torus link. Let $p\in \partial D$ be a point between $b_2$ and $b_3$ in the case of $D_n$, respectively between $b_3$ and $b_4$ in the case of $E_7$. Let $I$ be an arc in $D$ from a point of $\partial D$ between $b_1$ and $b_n$ to $p$. Finally, define $S$ to be the surface obtained from $S'$ by plumbing a positive Hopf band along $I$ below the surface $S'$. Denote the core curve of that plumbed Hopf band by $e_1$ (so $e_1\cap S'=I$). Each pair of consecutive bands $b_i,b_{i+1}$, $1\leq i\leq n$, gives rise to a closed curve $e_{i+1}$ that runs from $D$ to $D'$ through $b_i$ and back to $D$ through $b_{i+1}$.
\begin{figure}[h]
\includegraphics{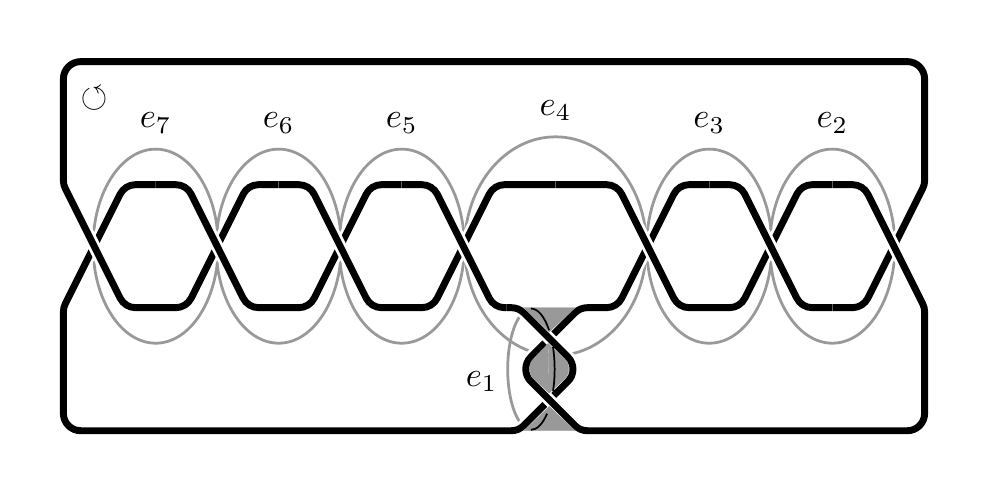}
\caption{$E_7$ fibre surface with homology basis coming from the plumbing tree. $\phi$ is the product of the right handed Dehn twists on the cuves $e_i$. \label{fig:E7Surface}}
\end{figure}
\begin{figure}[h]
\includegraphics{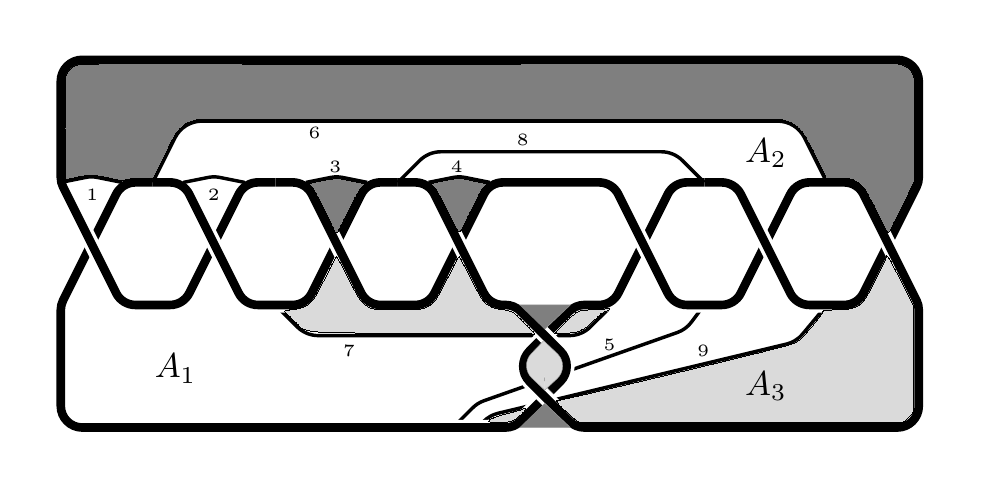}
\caption{Decomposition of the surface into three hexa\-gons $A_1,A_2,A_3$. Hexagon $A_3$ is shaded grey. The monodromy permutes the intervals $k_i$ (marked $1,2,\ldots,9$) cyclically. \label{fig:E7Intervals}}
\end{figure}
The incidence graph for the system of curves $e_1,\ldots, e_n$ in $S$ is exactly the respective Coxeter-Dynkin tree $E_7$ or $D_n$ (compare Figure~\ref{fig:E7Surface}).
The $e_i$ are core curves of positive Hopf bands and $S$ is a tree-like positive Hopf plumbing according to the respective tree. In particular, the monodromy $\phi$ of $S$ is the product of the right handed Dehn twists about the curves $e_2,e_3,\ldots,e_n,e_1$, in this order.
Just as in the case of torus links, we will find a finite number of disjoint arcs in $S$ that are permuted (up to free isotopy) by $\phi$ and such that these arcs cut $S$ into polygons. For $E_7$, let $k_1$ be the spanning arc of $b_7$, and let $k_{i+1} = \phi^i(k_1)$, $i=1,\ldots,8$, up to free isotopy (compare Figure~\ref{fig:E7Intervals}). Up to free isotopy, $\phi(k_9)=k_1$. This can be seen by applying the Dehn twists about the $e_j$ to the $k_i$, as described above. Another more visual way to see this is via {\em dragging arcs}. Imagine the arcs $k_i$ to be elastic bands whose ends are attached to the surface boundary and whose interiors are pushed slightly off the surface into the positive normal direction. Applying the monodromy $\phi$ amounts to dragging the arc through the complement of $S$ to the negative side of the surface, while its endpoints stay fixed on $\partial S$. Since we are only interested in the position of $\phi(k_i)$ up to free isotopy, the endpoints of the dragging arc may move freely along $\partial S$ during that process. Let $A_1,A_2,A_3$ be the three disk components of $S\setminus \bigcup_{i=1}^9 k_i$. The boundary of $A_j$ alternates between parts of $\partial S$ and the $k_i$. We choose the order as in Figure~\ref{fig:E7Hexagons}, where the components of $\partial A_j\cap \partial S$ are shrunk to points.
\begin{figure}[h]
\includegraphics{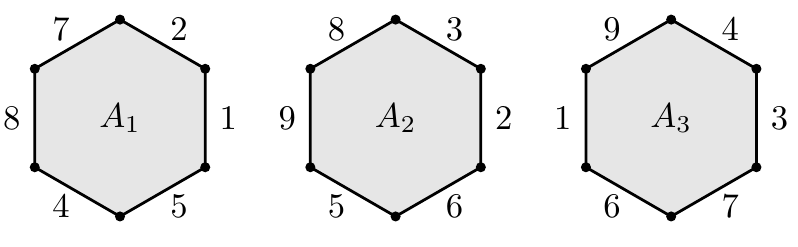}
\caption{Edges with the same label are glued. The monodromy sends $A_j$ to $A_{j+1}$, indices taken modulo $3$, such that edge $k_i$ is sent to edge $k_{i+1}$, modulo $9$. \label{fig:E7Hexagons}}
\end{figure}

\noindent Examination of the action of $\phi$ on the $k_i$ reveals that the $A_i$ are cyclically permuted by $\phi$, in the order $A_1\mapsto A_2\mapsto A_3\mapsto A_1$. In Figure~\ref{fig:E7Hexagons}, the $A_i$ are drawn in such a way that $A_1\mapsto A_2\mapsto A_3$ by translation to the right, and $A_3$ is mapped to $A_1$ by a translation, followed by a clockwise rotation through $1/3$. To obtain the {\tt} graph $\Gamma$, put a vertex in the middle of each hexagon $A_j$ and connect them by edges through the center of every $k_i$, connecting the vertices of the adjacent hexagons. The {\tt} twist lengths on the two boundary annuli are $1$ and $2$, respectively.

For the case of $D_n$, $n$ odd, take $k_1$ to be the spanning arc of $b_1$ and let $k_{i+1}=\phi^i(k_1)$, $i=1,\ldots,2n-3$. As before, we have $\phi(k_{2n-2})=k_1$, and the $k_i$ decompose $S$ into $n-1$ disks $A_1,\ldots,A_{n-1}$, as in Figure~\ref{fig:DnIntervalsOdd}. In Figure~\ref{fig:DnSquares} (top), $\phi$ maps $A_1\mapsto A_2\mapsto\cdots\mapsto A_{n-1}$ by right translations and sends $A_{n-1}$ back to $A_1$ by a rotation of $180^\circ$.

If $n$ is even, we use two orbits of intervals instead of one: define $k_1,\ldots,k_{n-1}$ and $k_1',\ldots,k_{n-1}'$ by letting $k_1,k_1'$ be the spanning arcs of $b_1$, $b_n$ respectively and $k_{i+1}=\phi^i(k_1)$, $k_{i+1}'=\phi^i(k_1')$. Again, the union of the $k_i$ and the $k_i'$ decomposes $S$ into disks $A_1,\ldots, A_{n-1}$ (see Figure~\ref{fig:DnIntervalsEven}). In Figure~\ref{fig:DnSquares} (bottom), the monodromy maps $A_1\mapsto A_2\mapsto\cdots\mapsto A_{n-1}\mapsto A_1$ by translations. The {\tt} graphs for $D_n$ have one vertex at the center of each square and edges pass through the $k_i$ and $k_i'$. Twist lengths on the boundary annuli are $1$, $n-2$ for odd $n$, and $1$, $2$, $\frac{n}{2}-1$ for even $n$.

\begin{figure}[h]
\includegraphics{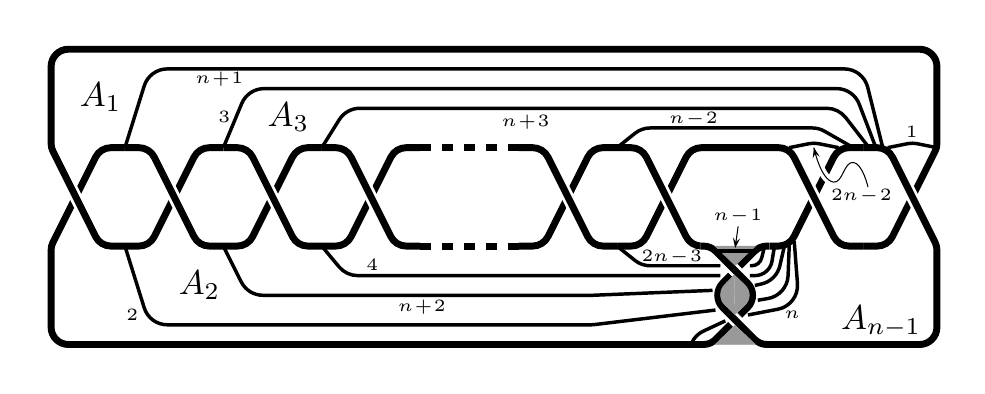}
\caption{Decomposing arcs $k_1,\ldots,k_{2n-2}$ on the fibre surface of $D_n$ for odd $n$.\label{fig:DnIntervalsOdd}}
\includegraphics{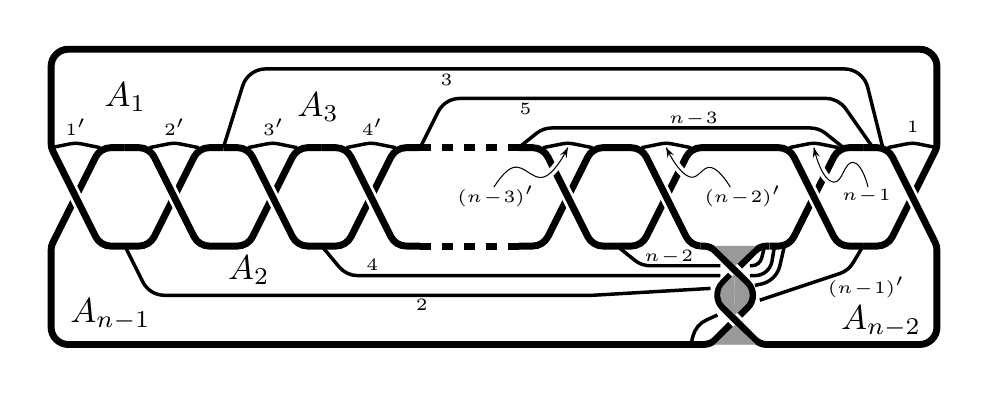}
\caption{Decomposing arcs $k_1,\ldots,k_{n-1},\ k_1',\ldots,k_{n-1}'$ on the fibre surface of $D_n$ for even $n$.\label{fig:DnIntervalsEven}}
\end{figure}

\begin{figure}[h]
\includegraphics{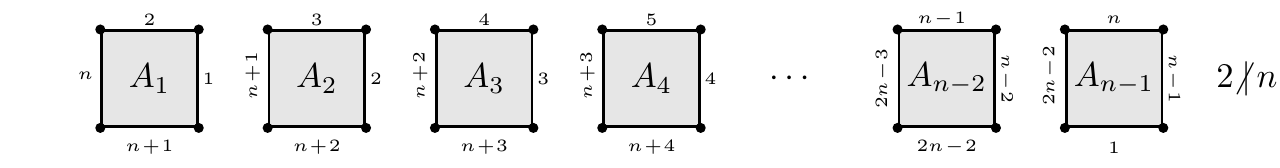}
\includegraphics{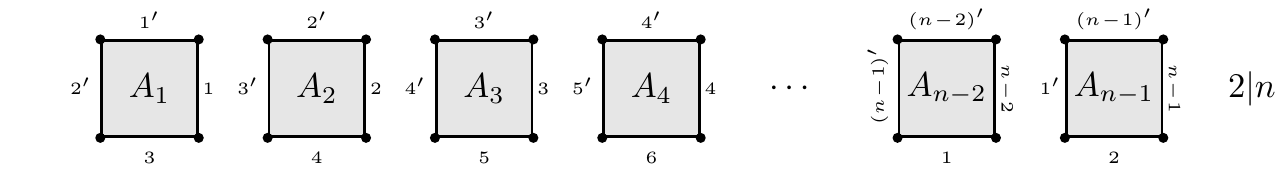}
\caption{Description of the monodromy of $D_n$, for odd $n$ (top) and for even $n$ (bottom).\label{fig:DnSquares}}
\end{figure}

\section{The finite cases.} \label{sec:exceptional}

In~\cite[Corollary~2]{BIRS}, Buck et al. show that $T(2,n)$ admits only finitely many arcs preserving fibredness (up to isotopy). More precisely, they show that every clean arc is isotopic (free on the boundary) to an arc that is contained in one of the disks $A_1,A_2$ from the above description of the monodromy of torus links. Apart from this infinite family of torus links, there are only three more torus links with just a finite number of arcs that preserve fibredness:

\begin{proposition} The torus links $T(3,3)$, $T(3,4)$ and $T(3,5)$ admit, up to isotopy (free on the boundary), only a finite number of cutting arcs that preserve fibredness.
\end{proposition}

For positive tree-like Hopf plumbed surfaces we similarly obtain:

\begin{proposition} The positive tree-like Hopf plumbings associated to any of the Coxeter-Dynkin trees $A_n$, $D_n$, $E_6$, $E_7$ or $E_8$ admit, up to isotopy (free on the boundary), only a finite number of cutting arcs that preserve fibredness.
\end{proposition}

The proofs of Propositions~\ref{prop:3345} and \ref{prop:DnE7} are rather technical and will be given in Section~\ref{sec:propproofs}. Nevertheless, the idea is very simple: let $S$ be the fibre surface of any torus link $T(n,m)$, given as thickening of a complete bipartite graph on $n+m$ vertices, or of $D_n$ or $E_7$, as described in Section~\ref{sec:monodromies}. An arc $\alpha\subset S$ is determined up to isotopy by its endpoints and by the sequence of bands $K$ it passes through. Now start listing all possible such sequences that yield clean arcs, for increasing length of the sequence. In order to prove finiteness of this list, we use three Lemmas, also given in Section~\ref{sec:propproofs}. The intuitive meaning of Lemma~\ref{lem:conflict} and Lemma~\ref{lem:generalisedconflict} can be phrased as follows: if $\alpha$ and $\phi(\alpha)$ intersect and this intersection seemingly cannot be removed by an isotopy, then $\alpha$ is indeed unclean. Lemma~\ref{lem:consecutive} asserts that a clean arc cannot stay in the complement of the graph for a distance of more than $\ell$ consecutive bands, where $\ell$ is the {\tt} twist length on the corresponding boundary annulus (for example, $\ell=2$ for all torus links).

This is made precise in Section~\ref{sec:propproofs}, using a notion of {\em arcs in normal position} (cf.\ Definition~\ref{def:normalpos}). Along with this case-by-case analysis, one can find all possible fibred links obtained from $A_{n-1} = T(2,n)$, $D_4=T(3,3)$, $D_n$, $E_6=T(3,4)$, $E_7$ and $E_8=T(3,5)$ by cutting along an arc. Consult Table~\ref{table} for a complete list.

\begin{table}[h]
\begin{tabular}{l|l}
 From     & one obtains by cutting along a clean arc \\ \hline\hline 
 $T(2,n)$ & $T(2,n-1),\quad T(2,m_1)\# T(2,m_2)$ for $m_1+m_2=n$\phantom{$A^{f^t}$} \\ \hline
 $T(3,3)$ & $T(2,4),\quad (T(2,2)\# T(2,2)\# T(2,2))^{*_1}$ \phantom{$A^{f^t}$} \\ \hline
 $T(3,4)$ & $D_5,\quad T(2,6),\quad T(2,5)\# T(2,2)$, \phantom{$A^{f^t}$} \\
          & $T(2,3)\# T(2,3)\# T(2,2),\quad (T(2,3)\# T(2,2)\# T(2,3))^{*_2}$\hspace{-4ex} \\ \hline
 $T(3,5)$ & $E_7,\quad D_7,\quad T(2,8),\quad (D_5\# T(2,3))^{*_3}$,\phantom{$A^{f^t}$} \\
          & $T(2,5)\# T(2,4),\quad T(2,7)\# T(2,2),\quad T(3,4)\# T(2,2)$,\hspace{-3ex} \\
	  & $T(2,5)\# T(2,3)\# T(2,2),\quad (T(2,5)\# T(2,2)\# T(2,3))^{*_4}$\hspace{-4ex} \\ \hline
 $D_n$    & $T(2,n),\quad D_{n-1},\quad D_{m_1}\# T(2,m_2)$ for $m_1+m_2=n$,\phantom{$A^{f^t}$} \hspace{-4ex} \\
          & $T(2,2)\# T(2,2)\# T(2,n-2)$ \\ \hline
 $E_7$    & $E_6,\quad D_6,\quad T(2,7),\quad T(2,4)\# T(2,2)\# T(2,3)$,\phantom{$A^{f^t}$} \hspace{-3ex} \\
          & $T(2,6)\# T(2,2),\quad T(2,5)\# T(2,3)$ \vspace{1ex} \\
\end{tabular}
\begin{flushleft}
$K_1\# K_2$ denotes the connected sum of $K_1$ and $K_2$, $D_n$ denotes the closure of the braid $\sigma_1^{n-2}\sigma_2\sigma_1^2\sigma_2$, $n\geq 3$, and $E_n$ denotes the closure of the braid $\sigma_1^{n-3}\sigma_2\sigma_1^3\sigma_2$, $n=6,7,8$.\\[2ex]
\begin{footnotesize} $^{*_1}$ chain of four successive unknots.\\
$^{*_2}$ both Hopf link components are summed to one trefoil knot each.\\
$^{*_3}$ both possible sums appear (trefoil summed with the unknot component of $D_5$ as well as trefoil summed with the trefoil component of $D_5$).\\
$^{*_4}$ one component of the Hopf link in the middle is summed to $T(2,5)$ and the other is summed to the trefoil.
\end{footnotesize}
\end{flushleft}
\vspace{2ex}
\caption{Fibred links obtained from the exceptional torus links by cutting along an arc. \label{table}}
\end{table}

\section{Arcs for links with infinite order monodromy} \label{sec:infinite}

\setcounter{theorem}{2}
\begin{theorem} \label{thm:3} Let $S$ be a surface obtained by iterated plumbing of positive Hopf bands and suppose the monodromy $\phi:S\to S$ is pseudo-Anosov. Then, $S$ contains infinitely many non-isotopic cutting arcs preserving fibredness.
\end{theorem}

\begin{proof} The monodromy $\phi$ is a composition of right Dehn twists along the core curves of the Hopf bands used for the construction of $S$ as a Hopf plumbing. Let $\alpha$ be an arc dual to the core curve of the last plumbed Hopf band and such that $\alpha$ does not enter any of the previously plumbed Hopf bands. Then, in the product of Dehn twists representing $\phi$, only the last factor affects $\alpha$. It follows that $\alpha$ is clean (and therefore $\phi^n(\alpha)$ is also clean by Remark~\ref{rem:clean}). Since $\phi$ is pseudo-Anosov and $\alpha$ is essential, the length of $\phi^n(\alpha)$ (with respect to an auxiliary Riemannian metric) grows exponentially as $n$ tends to infinity (compare \cite{FM}, Section~14.5). In particular, the arcs $\phi^n(\alpha)$ are pairwise non-isotopic and clean.
\end{proof}

In general, it is not enough to require $\phi$ to be non-periodic. Indeed, the family of arcs $\phi^n(\alpha)$ might be finite, even if $\phi$ is of infinite order. This occurs typically when $\phi$ is reducible and $\alpha$ is contained in a periodic reducible piece of $\phi$. However, if we dispose of a homology class $[c]\in H_1(S,\Z)$ whose coordinate dual to $\alpha$ grows (i.e. $i(\phi^n(c),\alpha)\to\infty$ for $n\to\infty$), then the family $\phi^{-n}(\alpha)$ contains infinitely many distinct arcs since $i(\phi^n(c),\alpha) = i(c,\phi^{-n}(\alpha))\to\infty$.

\begin{proposition} \label{prop:infinitetrees} Let $S$ be a surface obtained by plumbing positive Hopf bands according to a tree, other than $A_n$, $D_n$, $E_6$, $E_7$ and $E_8$. Then, $S$ contains infinitely many non-isotopic cutting arcs preserving fibredness.
\end{proposition}

\begin{proof}
Let $S$ be a surface obtained by positive tree-like Hopf plumbing. Denote the induced action of the monodromy on $H_1(S,\Z)$ by $\phi_*$, and let $e_1,\ldots,e_n\in H_1(S,\Z)$ be the basis vectors represented by the core curves of the Hopf bands used in the plumbing construction. It follows from A'Campo's work on the spectrum of Coxeter transformations (\cite{AC1}) and slalom knots (\cite{AC2}), that $\phi_*$ has a real eigenvalue $\lambda$ with $|\lambda|>1$ if the tree corresponds to neither spherical nor affine Coxeter systems. Let $c$ be an eigenvector of $\phi_*$ for the eigenvalue $\lambda$. Then the sequence $\phi_*^n(c)=\lambda^n c$ is unbounded. Choose $j\in\{1,\ldots,n\}$ such that the $j$-th coordinate of $c$ is nonzero. Let $\alpha\subset S$ be a spanning arc of the Hopf band with core curve $e_j$. Then $\alpha$ is clean, since cutting along $\alpha$ yields a connected sum of positive tree-like Hopf plumbings, which is fibred. Moreover we have $|i(\phi^n(c),\alpha)|\to\infty$ for $n\to\infty$ by construction. It therefore remains to study the affine Coxeter-Dynkin trees. For these, the spectral radius of $\phi_*$ is equal to one. However, $\phi_*$ has a Jordan block to the eigenvalue $-1$ in these cases, and a similar reasoning applies.
\end{proof}

\begin{proof}[Proof of Theorem~\ref{thm:2}]
Combine Propositions~\ref{prop:DnE7} and~\ref{prop:infinitetrees}.
\end{proof}

\section{Proof of Propositions~\ref{prop:3345} and~\ref{prop:DnE7}} \label{sec:propproofs}

\setcounter{proposition}{0}
\begin{proposition} \label{prop:3345} The torus links $T(3,3)$, $T(3,4)$ and $T(3,5)$ admit, up to isotopy (free on the boundary), only a finite number of cutting arcs that preserve fibredness.
\end{proposition}

\begin{proposition} \label{prop:DnE7} The positive tree-like Hopf plumbings associated to any of the Coxeter-Dynkin trees $A_n$, $D_n$, $E_6$, $E_7$ or $E_8$ admit, up to isotopy (free on the boundary), only a finite number of cutting arcs that preserve fibredness.
\end{proposition}

Before we begin with the proofs, some notation and remarks are necessary. Let $S$ be the fibre surface of either $T(n,m)$, $D_n$ or $E_7$, and let $\phi:S\to S$ be the monodromy. Precisely as in Section~\ref{sec:monodromies}, we decompose $S$ into finitely many disjoint polygonal disks $A_i$ (and $B_j$ in the case of torus links) that are glued using bands ($K_{ij}$ for the torus links and neighbourhoods of the $k_i$, $k_i'$ for $D_n$ and $E_7$). We use the letter $D$ to denote any of the disks and the letter $K$ to denote any of the bands. Let $U$ be the union of all the disks, and let $N\subset S$ be the neighbourhood of the {\tt} graph on which $\phi$ is assumed to be periodic.

\begin{definition} \label{def:normalpos}
An arc $\alpha\subset S$ is in {\em normal position} if the following conditions hold:
\begin{enumerate}[leftmargin=*]
 \item The endpoints of $\alpha$ lie in $\partial U$.
 \item \label{def:normalpos:b} For every band $K$, $\alpha\cap K\setminus U$ consists of finitely many straight segments parallel to the edges of the {\tt} graph.
 \item \label{def:normalpos:c} The number of such segments in $K$ is minimal among all arcs isotopic to $\alpha$.
 \item $\alpha$ intersects the graph transversely in finitely many points of $U$.
 \item $\alpha\setminus N \subset U$, that is, before $\alpha$ enters $N$ and after it leaves $N$, it stays in the disks that contain its endpoints.
 \item \label{def:normalpos:f} $\alpha\cap U$ consists of finitely many straight arcs.
\end{enumerate}
\end{definition}

\setcounter{remarks}{\value{remark}}
\begin{remarks}[on normal position] \label{rem:normalpos} \hfill
\begin{itemize}[leftmargin=*]
 \item Any arc can be brought into normal position by a free isotopy.
 \item If $\alpha$ is in normal position, then $\phi(\alpha)$ can be brought into normal position keeping $N$ fixed. Indeed, it suffices to straighten the two subarcs $\phi(\alpha)\setminus N$ (or, undoing the twisting that occurs in the respective annuli), sliding the endpoints of $\phi(\alpha)$ along $\partial S$, see Figure~\ref{fig:slide}.

\begin{figure}[h]
\includegraphics{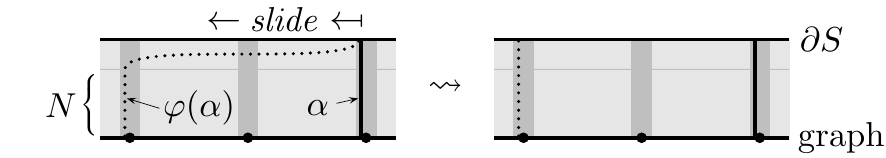}
\caption{How to bring $\phi(\alpha)$ in normal position, keeping $N$ fixed.\label{fig:slide}}
\end{figure}

 \item If $\alpha$ and $\phi(\alpha)$ are in normal position as above, we may isotope $\phi(\alpha)$ with endpoints fixed and keeping it in normal position, such that $\alpha$ and $\phi(\alpha)$ intersect transversely in finitely many points of $U$. In particular, the sets $\alpha\setminus U$ and $\phi(\alpha)\setminus U$ are now disjoint (cf.\ Figure~\ref{fig:transverse}).

\begin{figure}[h]
\includegraphics{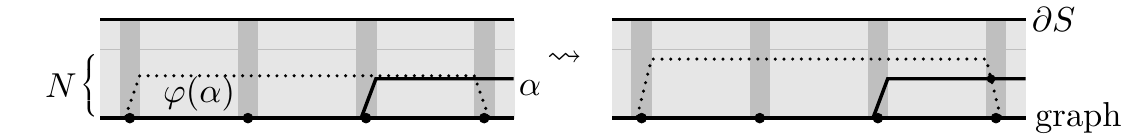}
\caption{How to make $\alpha,\phi(\alpha)$ intersect transversely, keeping normal position.\label{fig:transverse}}
\end{figure}

 \item Let $\alpha$ be in normal position and suppose it passes through at least one band. Let $K$ be the first (respectively last) band traversed by $\alpha$ after (before) it starts (ends) at a boundary point $p$ of one of the disks, say $D$. Then, $p$ cannot lie between $K$ and one of the two bands adjacent to $K$ on $\partial D$. Otherwise, an isotopy sliding the starting point (endpoint) of $\alpha$ along $\partial K$ would decrease the number of segments in $K$, contradicting part~(\ref{def:normalpos:c}) of Definition~\ref{def:normalpos}.
\end{itemize}
\end{remarks}

\begin{remarks}[compare the {\em bigon criterion}, Prop.~1.7 in~\cite{FM}] \label{rem:bigons}
Suppose $\alpha$ intersects $\phi(\alpha)$. If $\alpha$ is clean, there must be a bigon $\Delta\subset S$ whose sides consist of a subarc of $\alpha$ and a subarc of $\phi(\alpha)$. If $\alpha,\phi(\alpha)$ are in normal position, such $\Delta$ takes a particularly simple form:
\begin{itemize}[leftmargin=*]
 \item $\Delta$ cannot be contained in $U$ (i.e., in one of the disks $A_i$ or $B_j$). This would contradict part~(\ref{def:normalpos:f}) of the above Definition~\ref{def:normalpos}.
 \item None of the two sides of $\Delta$ is contained in $U$, since the other side of $\Delta$ would have to leave $U$ through one of the bands $K$ and return through the same $K$. The disk $\Delta$ would then yield an isotopy reducing the number of segments of $\alpha\cap K$ or $\phi(\alpha)\cap K$, contradicting part~(\ref{def:normalpos:c}) of Definition~\ref{def:normalpos}.
 \item For every band $K$, $\Delta\cap K\setminus U$ consists of rectangles with two opposite sides parallel to the edge passing through $K$.
 \item $\Delta\cap U$ constists of topological disks $\delta$ connected to at least one rectangle.
 \item Construct a spine $T$ for $\Delta$ as follows: put a vertex for each $\delta$ and connect two vertices by an edge if the corresponding disks $\delta$ connect to the same rectangle. $T$ is a tree, for $\Delta$ is contractible. Two of its vertices correspond to the vertices of the bigon $\Delta$. Among the other vertices of $T$, there is none of degree one because the adjacent edge would correspond to a rectangle in some $K$ whose sides parallel to its core edge both belong to the same arc ($\alpha$ or $\phi(\alpha)$). In other words, either $\alpha$ or $\phi(\alpha)$ would pass through $K$ and immediately return through $K$ in the opposite direction. This would contradict part~(\ref{def:normalpos:c}) of Definition~\ref{def:normalpos}. Therefore, $T$ is a line consisting of some number of consecutive edges, and the two extremal vertices correspond to the vertices of $\Delta$.
\end{itemize}
\end{remarks}

\begin{lemma} \label{lem:conflict} Let $\alpha,\phi(\alpha)$ be in normal position and suppose they intersect in a point $p\in D$, where $D$ is one of the disks $A_i$ (or $B_j$ in the torus link case). Let $\alpha',\alpha''$ be the components of $\alpha\cap D, \phi(\alpha)\cap D$ containing $p$. If no two of the four points $\partial\alpha'\cup\partial\alpha''\subset\partial D$ lie in the same band $K$, then $\alpha$ cannot be clean.
\end{lemma}

\begin{figure}[h]
\includegraphics{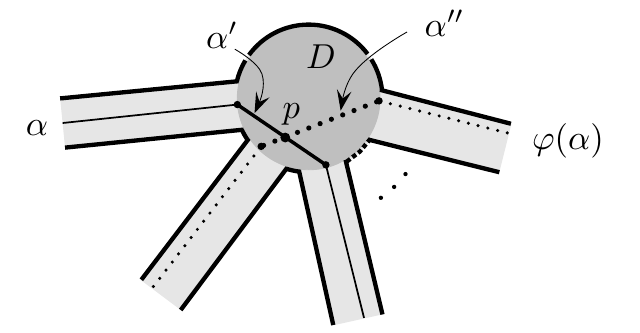}
\caption{$\alpha$ cannot be clean by Lemma~\ref{lem:conflict}.\label{fig:conflict}}
\end{figure}

\setcounter{remark}{\value{remarks}}
\begin{remark} \label{rem:conflict} Note that we did not exclude the possibility that one of the endpoints of $\alpha$ or $\phi(\alpha)$ lie in $\partial\alpha'\cup\partial\alpha''$.
\end{remark}

\begin{proof}[Proof of Lemma~\ref{lem:conflict}]
If $\alpha$ were clean, there would be a bigon. After possibly removing a certain number of such bigons, we are left with a bigon $\Delta$ with vertex $p$. By Remark~\ref{rem:bigons}, $\Delta$ has to leave $D$ through one of the adjacent bands. Since one of the sides of $\Delta$ is a subarc of $\alpha$ and the other side is a subarc of $\phi(\alpha)$, we find two points among $\partial\alpha'\cup\partial\alpha''$ that lie in this band, contradicting the assumption on $\alpha',\alpha''$. 
\end{proof}

\begin{figure}[h]
\includegraphics{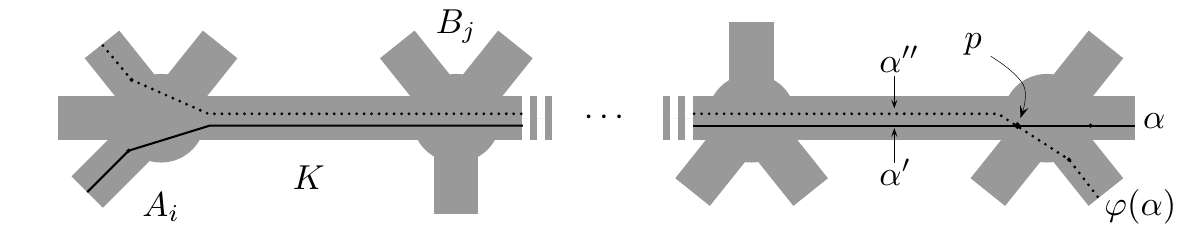}
\caption{$\alpha$ cannot be clean by Lemma~\ref{lem:generalisedconflict}.\label{fig:generalisedconflict}}
\end{figure}

\begin{lemma} \label{lem:generalisedconflict}
Let $\alpha,\phi(\alpha)$ be in normal position and let $\alpha',\alpha''$ be subarcs of $\alpha,\phi(\alpha)$ respectively (not necessarily contained in $U$). Suppose that the four endpoints of $\alpha'$ and $\alpha''$ are contained in $\partial U$ and that no two of them lie on the same band $K$. We further assume that $\alpha'$ and $\alpha''$ intersect in exactly one point and that $\alpha',\alpha''$ run through the same bands (see Figure~\ref{fig:generalisedconflict}). Then $\alpha$ cannot be clean.
\end{lemma}

\begin{proof}
Assume $\alpha'\cap\alpha''=\{p\}$, then $p\in U$. As in the proof of Lemma~\ref{lem:conflict}, study a bigon $\Delta$ that starts at $p$. $\Delta$ consists of a sequence of rectangles as described in Remarks~\ref{rem:bigons}. Starting at $p$, it therefore has to pass through the same bands as $\alpha'$ and $\alpha''$. Since $p$ was the only intersection between $\alpha'$ and $\alpha''$, $\Delta$ has to pass through at least one more band. But this is impossible by the assumption on the endpoints of $\alpha'$ and $\alpha''$.
\end{proof}

\begin{lemma} \label{lem:consecutive} A clean arc in normal position cannot traverse more than $\ell$ consecutive bands along a complementary annulus of twist length $\ell$.
\end{lemma}

Here, a sequence of bands $K^{(1)},K^{(2)},\ldots$ is {\em consecutive}, if the set $(\bigcup_r K^{(r)}\cup\bigcup_i A_i\cup\bigcup_j B_j)\setminus\Gamma$ has a connected component that intersects all bands $K^{(r)}$ of the sequence in this order, i.e.\ it is possible to stay on the same side of the graph when walking along the bands. The twist length $\ell$ denotes the number of edges of $\Gamma$ enclosed between $\gamma$ and $\phi(\gamma)$, where $\gamma$ is a spanning arc of the corresponding boundary annulus that ends at a vertex of $\Gamma$ (compare Section~\ref{sec:monodromies}).

\begin{proof} Suppose that $\alpha$ is a clean arc in normal position that traverses $n$ consecutive bands, $n\geq\ell+1$. We may assume that $n$ is the maximal number of consecutively traversed bands. In these bands as well as the adjacent disks, isotope $\alpha$ such that it stays on one side of the graph, keeping it in normal position. Now bring $\phi(\alpha)$ into normal position transverse to $\alpha$ as described in the Remarks~\ref{rem:normalpos}. Recall the description of the monodromy $\phi$ as a {\tt} twist from Section~\ref{sec:monodromies}: cutting the surface $S$ open along the graph results in $d$ annuli, where $d$ is the number of components of $\partial S=L$ and each annulus has a link component as one boundary and a cycle consisting of edges of the graph as the other boundary. In one of these annuli we will see a subarc $\alpha'\subset\alpha$ that has exactly its endpoints in common with the graph and that travels near the edge boundary for a distance of $n$ consecutive edges. (Note that $\alpha'$ cannot have any endpoint on $\partial S$. This would contradict part~(\ref{def:normalpos:c}) of Definition~\ref{def:normalpos}). Let $C$ be the disk bounded by $\alpha'$ and the graph. The monodromy $\phi$ keeps the link-boundary of this annulus fixed and rotates the neighbourhood $N$ of the graph boundary by $\ell$ edges. Since $n\geq\ell+1$, $\phi(\alpha')$ has one of its endpoints in $C$ and the other outside of $C$, so $\alpha'$ has to intersect its image $\phi(\alpha')$ in a point $p\in U$, and we may assume that $p$ is the only intersection between $\alpha'$ and $\phi(\alpha')$. Denote by $q$ the endpoint of $\phi(\alpha')$ that lies in $C$ and let $D$ be the disk $A_i$ or $B_j$ containing $q$. Then make sure that $p\in D$ by an isotopy on $\phi(\alpha')$ preserving normal position if necessary (compare Figures~\ref{fig:Annulus} and~\ref{fig:crossing}).
\begin{figure}[h]
\includegraphics{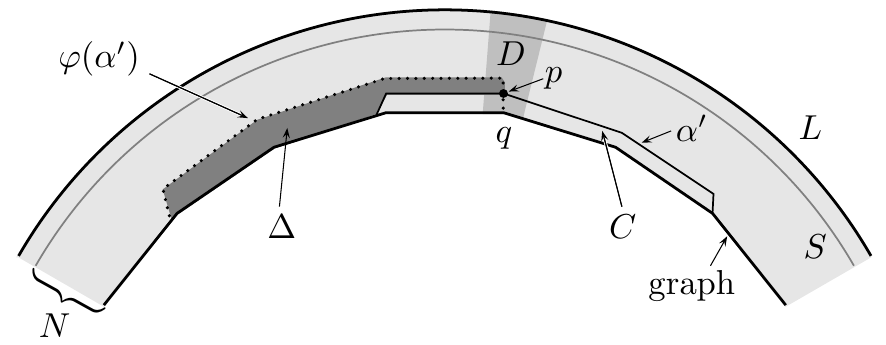}
\caption{A normal arc passing through more than $\ell$ consecutive bands has to intersect its image under the monodromy (here $\ell=2$). Part of an a priori possible bigon $\Delta$. \label{fig:Annulus}}
\end{figure}
However $\alpha$ is clean, so there must be a bigon in $S$ whose sides consist of a subarc of $\alpha$ and a subarc of $\phi(\alpha)$. After possibly removing a certain number of such bigons, we will be left with a bigon $\Delta$ starting at $p$. From the Remarks~\ref{rem:bigons} we know that $\Delta$ has to leave $D$ and consists of a sequence of rectangles. Let $R$ be the first rectangle in this sequence, i.e.\ $R$ is contained in a band adjacent to $D$. Let $K^-,K^+$ be the two bands adjacent to $D$ that contain segments of $\alpha'$, $K^+$ being the one that also contains a segment of $\phi(\alpha')$ (see Figure~\ref{fig:crossing}).
\begin{figure}[h]
\includegraphics{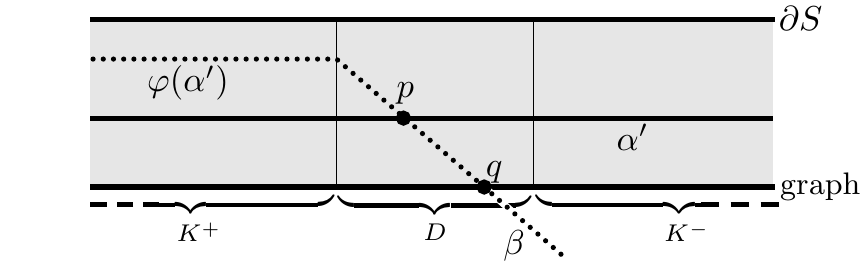}
\caption{Part of the mentioned annulus, where the arcs $\alpha'$ and $\phi(\alpha')$ intersect in a point $p\in D$.\label{fig:crossing}}
\end{figure}
Let $\beta$ be the component of $\phi(\alpha)\setminus\{p\}$ that contains $q$. We claim that $\beta$ cannot leave $D$ through $K^-$ nor $K^+$. Indeed, if $\beta$ would leave $D$ through $K^-$, $\phi(\alpha)$ would traverse $n+1$ consecutive bands, contradicting the assumption on $n$ being maximal. On the other hand, if $\beta$ would leave $D$ through $K^+$, we could reduce the number of segments in $\phi(\alpha)\cap K^+$, contradicting the normal position of $\phi(\alpha)$, i.e.\ part~(\ref{def:normalpos:c}) of Definition~\ref{def:normalpos}. In contrast, $\alpha$ leaves $D$, starting from $p$ in both directions, through $K^-$ and $K^+$. Consider now the subarcs of $\alpha$ and $\phi(\alpha)$ that constitute two opposite sides of the rectangle $R$. Since $R$ is contained in a band adjacent to $D$, these two subarcs arrive at $\partial D$ through the same band, and they connect directly to $p\in D$. Therefore, we must have $R\subset K^+$, since $K^+$ is the only band containing two subarcs of $\alpha$ and $\phi(\alpha)$ that directly connect to $p\in D$. Furthermore, $R$ has to be the region enclosed by $\alpha'\cap K^+$ and $\phi(\alpha')\cap K^+$. Following $\phi(\alpha')$ in the direction from $q$ to $p$, we see that it leaves $D$ through $K^+$ as one of the sides of $R$ and continues staying on the same side of the graph for exactly $n-1$ more edges. By assumption, $p$ is the only intersection between $\alpha'$ and $\phi(\alpha')$, so the bigon $\Delta$ has to continue for at least $n-1$ more rectangles through consecutive bands. Similarly, the sides of these rectangles that are subarcs of $\alpha$ have to continue for at least $n-1$ consecutive bands. We obtain a contradiction to the maximality of $n$, because $\alpha'$ ends after $n-\ell$ bands starting from $D$, since $\phi$ rotates the graph by $\ell$ edges. This finishes the proof.
\end{proof}

\begin{proof}[Proof of Proposition \ref{prop:3345}]
We will concentrate on the most complicated case of the torus knot $T(3,5)$. It contains all difficulties appearing in the proofs for $T(3,3)$ and $T(3,4)$ which go along the same lines with fewer cases to consider. For each link appearing in Table~\ref{table} of Section~\ref{sec:exceptional}, we will indicate one (but not every) possible choice of a cutting arc that yields the link in question. Let hence $S$ be the fibre surface of $T(3,5)$ and let $\alpha\subset S$ be any arc that preserves fibredness, i.e.\ a clean arc. Bring $\alpha$ into normal position using an isotopy (not fixing the boundary), cf.\ Remarks~\ref{rem:normalpos}. Since $\phi$ permutes the vertices $\{a_i\}$ cyclically as well as the vertices $\{b_j\}$, it suffices to show that there are only finitely many clean arcs starting at a point of $\partial A_1$ or at a point of $\partial B_1$, up to isotopy. We may further assume that $\alpha$ starts either at a point of $\partial A_1$ between $k_{11}$ and $k_{15}$ or at a point of $\partial B_1$ between $k_{21}$ and $k_{31}$.

{\em \ul{Case A.}} $\alpha$ starts at $\partial A_1$, between $k_{11}$ and $k_{15}$. Then, $\alpha$ cannot continue through either of the bands $K_{11}$ nor $K_{15}$ by the last item of Remarks~\ref{rem:normalpos}. So, either $\alpha$ stays in $A_1$ (and there are only four such arcs up to isotopy), or it continues through $K_{12}, K_{13}$ or $K_{14}$. If $\alpha$ stays in $A_1$, the links obtained by cutting are $E_7$ (e.g.\ if $\alpha$ ends between $k_{11}$ and $k_{12}$) and $D_7$ (e.g.\ if $\alpha$ ends between $k_{12}$ and $k_{13}$).

{\em \ul{Case A.1.}} $\alpha$ continues through $K_{12}$. Arriving in $B_2$, there are three possibilities: either $\alpha$ ends at a point of $\partial B_2$ between $k_{22}$ and $k_{32}$ (and cutting along $\alpha$ yields $T(3,4)\# T(2,2)$), or it continues through $K_{22}$ or $K_{32}$ (ending at other points of $\partial B_2$ is impossible by the last item of Remarks~\ref{rem:normalpos}).

{\em \ul{Case A.1.1.}} $\alpha$ continues through $K_{22}$. Arriving in $A_2$, $\alpha$ can end at a point of $\partial A_2$ (cutting yields $T(2,7)\# T(2,2)$ if $\alpha$ ends between $k_{24}$ and $k_{25}$, and $T(2,3)$ summed with the unknot component of $D_5$ if $\alpha$ ends between $k_{23}$ and $k_{24}$), or it can continue through $K_{23}$ or $K_{24}$. It cannot continue through $K_{21}$, since $K_{12},K_{22},K_{21}$ is a sequence of three consecutive bands, so $\alpha$ would not be clean by Lemma~\ref{lem:consecutive}. Finally, $\alpha$ cannot continue through $K_{25}$. If it did, $\alpha$ and $\phi(\alpha)$ would intersect in a point of $A_1$, and Lemma~\ref{lem:conflict} would imply that $\alpha$ cannot be clean (see Figure~\ref{fig:T35}, top left). Note that we do not know whether the mentioned intersection is the only one since we do not know how $\alpha$ ends.

{\em \ul{Case A.1.1.1.}} $\alpha$ continues through $K_{23}$. From $B_3$, it cannot continue through $K_{13}$, for $K_{22},K_{23},K_{13}$ are consecutive (Lemma~\ref{lem:consecutive}). If it continues through $K_{33}$ it cannot continue through any band adjacent to $A_3$. Indeed, $K_{23},K_{33},K_{32}$ are consecutive, so $\alpha$ cannot continue through $K_{32}$. If it would continue through $K_{34}$ or $K_{35}$ or $K_{31}$, we could apply Lemma~\ref{lem:generalisedconflict} to the band $K_{33}$ to show that $\alpha$ is not clean (see Figure~\ref{fig:T35}).

\begin{figure}[ht]
\includegraphics{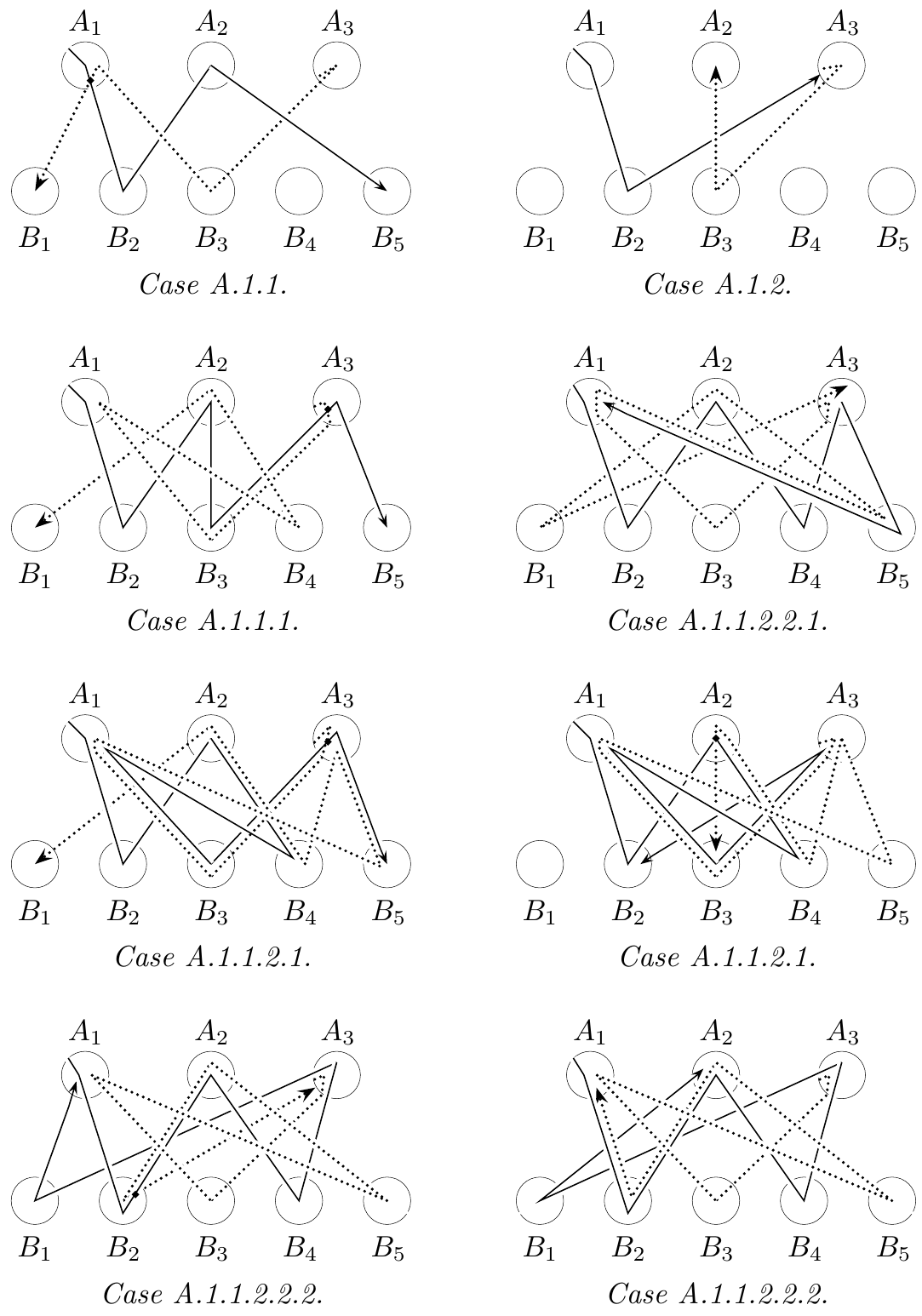}
\caption{Schematic illustration for a selection of the cases in the proof of Proposition~\ref{prop:3345}. The arc $\alpha$ is drawn as solid line, whereas $\phi(\alpha)$ is shown as a dotted line. \label{fig:T35}}
\end{figure}

{\em \ul{Case A.1.1.2.}} $\alpha$ continues through $K_{24}$. If it ends in $B_4$ between $k_{14}$ and $k_{34}$, we obtain $T(2,5)\# T(2,4)$ after cutting. Otherwise, it can continue from $B_4$ through $K_{14}$ or through $K_{34}$.

{\em \ul{Case A.1.1.2.1.}} If it continues through $K_{14}$, it cannot go further. Firstly, $K_{24},K_{14},K_{15}$ are consecutive, so $K_{15}$ is no option (Lemma~\ref{lem:consecutive}). Neither can it proceed through $K_{11}$ (this would produce a self-inter\-sec\-tion of $\alpha$) nor $K_{12}$ (for otherwise we could apply Lemma~\ref{lem:conflict} to an intersection between $\alpha$ and $\phi(\alpha)$ in $A_1$). If it continues through $K_{13}$, it cannot go on through $K_{23}$ since $K_{14},K_{13},K_{23}$ are consecutive (Lemma~\ref{lem:consecutive}). Suppose it continues through $K_{33}$. From $A_3$, it cannot proceed through any of $K_{31},K_{35},K_{34}$, for otherwise we could apply Lemma~\ref{lem:generalisedconflict} to the bands $K_{13}$ and $K_{33}$, with an intersection between $\alpha$ and $\phi(\alpha)$ occuring in $A_3$ (see Figure~\ref{fig:T35} left). However, $\alpha$ cannot continue through $K_{32}$ either, because we could again apply Lemma~\ref{lem:generalisedconflict}, this time for the band $K_{24}$ and an intersection in $A_2$ (see Figure~\ref{fig:T35} right).

{\em \ul{Case A.1.1.2.2.}} $\alpha$ continues from $B_4$ through $K_{34}$. If it ends in $A_3$ between $k_{35}$ and $k_{31}$, cutting yields $T(2,5)\# T(2,2)\# T(2,3)$. Otherwise, it cannot continue from $A_3$ through $K_{33}$ since $K_{24},K_{34},K_{33}$ are consecutive. Neither can it proceed through $K_{32}$ (apply Lemma~\ref{lem:conflict} to $A_3$). So $\alpha$ can only continue through $K_{35}$ or $K_{31}$.

{\em \ul{Case A.1.1.2.2.1.}} If it continues through $K_{35}$, the only option to go further is through $K_{15}$, since $K_{34},K_{35},K_{25}$ are consecutive. From $A_1$ (compare Figure~\ref{fig:T35}), it cannot continue through $K_{11}$ nor $K_{12}$ (apply Lemma~\ref{lem:generalisedconflict} to $K_{15}$ with an intersection occuring in $A_1$). Neither can it continue through $K_{14}$, since $K_{35},K_{15},K_{14}$ are consecutive. So it has to go through $K_{13}$. Arriving in $B_3$, it cannot continue through $K_{23}$ (apply Lemma~\ref{lem:generalisedconflict} to $K_{34}$ with an intersection occuring in $B_4$). Therefore $\alpha$ has to continue through $K_{33}$. From $A_3$, it cannot proceed further. Firstly, $K_{32}$ is not an option (otherwise apply Lemma~\ref{lem:generalisedconflict} to $K_{34}$ and $K_{24}$ with an intersection in $A_2$). Neither can it go through $K_{34}$ or $K_{35}$ (apply Lemma~\ref{lem:generalisedconflict} to $K_{15},K_{13},K_{33}$ with an intersection occuring in $A_3$). Finally, it cannot pass through $K_{31}$ either (apply Lemma~\ref{lem:generalisedconflict} to the bands $K_{15},K_{13},K_{33}$ with an intersection occuring in $A_3$).

{\em \ul{Case A.1.1.2.2.2.}} If it continues through $K_{31}$ and arrives in $B_1$, it cannot proceed through $K_{11}$ (apply Lemma~\ref{lem:generalisedconflict} to $K_{22}$ with an intersection occuring in $B_2$, see Figure~\ref{fig:T35} left). So it has to go through $K_{21}$. From $A_2$, it cannot proceed through $K_{22}$, for $K_{31},K_{21},K_{22}$ are consecutive. Neither can it go through either of $K_{23}$ nor $K_{24}$ (apply Lemma~\ref{lem:conflict} to an intersection occuring in $A_2$, see Figure~\ref{fig:T35} right). Finally, $K_{25}$ can be ruled out by Lemma~\ref{lem:generalisedconflict}, applied to the bands $K_{22}$ and $K_{12}$, with an intersection occuring in $A_1$.

{\em \ul{Case A.1.2.}} $\alpha$ continues through $K_{32}$ (see Figure~\ref{fig:T35}). Arriving in $A_3$, it cannot continue through any band. Firstly, $K_{12},K_{32},K_{33}$ are consecutive, so $\alpha$ cannot continue through $K_{33}$. If it would continue through any of the other bands adjacent to $A_3$, $\alpha$ would intersect $\phi(\alpha)$ in a point of $A_3$ such that we could apply Lemma~\ref{lem:conflict} to obtain a contradiction to $\alpha$ being clean.

{\em \ul{Case A.2.}} $\alpha$ proceeds through $K_{13}$. If it ends in $B_3$ between $k_{23}$ and $k_{33}$, we obtain $T(2,8)$ after cutting. From $B_3$, it can continue through $K_{23}$ or through $K_{33}$.

{\em \ul{Case A.2.1.}} $\alpha$ continues through $K_{23}$. It cannot go on via $K_{22}$, for $K_{13},K_{23},K_{22}$ are consecutive. Neither can it continue through $K_{21}$ or $K_{25}$ by Lemma~\ref{lem:conflict} applied to an intersection in $A_1$. If it next passes through $K_{24}$, it cannot go on through $K_{14}$, because $K_{23},K_{24},K_{14}$ are consecutive. Proceeding through $K_{34}$, it can end in $A_3$ between $k_{31}$ and $k_{32}$ (this yields $T(2,5)\# T(2,3)\# T(2,2)$). However, the only possibility for $\alpha$ to go further is via $K_{32}$, for $K_{24},K_{34},K_{33}$ are consecutive (so $K_{33}$ is no option), and $\alpha$ cannot continue through $K_{35}$ nor $K_{31}$ by applying Lemma~\ref{lem:generalisedconflict} to the band $K_{34}$ with an intersection of $\alpha,\phi(\alpha)$ in $A_3$. So $\alpha$ continues through $K_{32}$ and arrives in $B_2$. From there, it cannot continue through $K_{12}$ (apply Lemma~\ref{lem:generalisedconflict} to $K_{22}$ and an intersection in $B_3$). If it continues through $K_{22}$, it cannot go further: $K_{23}$ is impossible because $K_{32},K_{22},K_{23}$ are consecutive, $K_{24}$ can be excluded by Lemma~\ref{lem:conflict}, applied to $A_2$, and $K_{21}$ as well as $K_{25}$ can be ruled out by Lemma~\ref{lem:generalisedconflict}, applied to $K_{23}$ and $K_{13}$ with an intersection occuring in $A_1$.

{\em \ul{Case A.2.2.}} $\alpha$ continues through $K_{33}$. This is similar to Case~A.2.1. Again there is always a single option to go on, until there is no possibility left after four more steps.

{\em \ul{Case A.3.}} $\alpha$ continues through $K_{14}$. This is analogous to Case~A.1.\\

{\em \ul{Case B.}} $\alpha$ starts at $\partial B_1$ between $k_{21}$ and $k_{31}$. Then, it can only continue through $K_{11}$ by the last item of Remarks~\ref{rem:normalpos}. From $A_1$, it can proceed through four distinct bands.

{\em \ul{Case B.1.}} $\alpha$ continues through $K_{15}$. Since $K_{11},K_{15},K_{25}$ are consecutive, it can a priori only continue through $K_{35}$. But this is impossible as well by Lemma~\ref{lem:conflict}, applied to the intersection between $\alpha$ and $\phi(\alpha)$ occuring in $B_1$.

{\em \ul{Case B.2.}} $\alpha$ continues through $K_{12}$. This is analogous to Case~B.1.

{\em \ul{Case B.3.}} $\alpha$ continues through $K_{14}$. Arriving in $B_4$, it can end between $k_{24}$ and $k_{34}$ (this results in $T(2,3)$ summed with the trefoil component of $D_5$).

{\em \ul{Case B.3.1.}} $\alpha$ continues through $K_{24}$. From $A_2$, it cannot continue through $K_{23}$ because $K_{14},K_{24},K_{23}$ are consecutive (Lemma~\ref{lem:consecutive}). Neither can it go on through $K_{22}$ nor $K_{21}$ (apply Lemma~\ref{lem:conflict} to $A_1$). Suppose $\alpha$ continues through $K_{25}$. From $B_5$, it cannot go on via $K_{35}$ since $K_{24},K_{25},K_{15}$ are consecutive. If it proceeds via $K_{35}$, we can apply Lemma~\ref{lem:generalisedconflict} to the band $K_{11}$ with an intersection in $B_1$ to obtain a contradiction.

{\em \ul{Case B.3.2.}} $\alpha$ continues through $K_{34}$. From $A_3$, there are only two options for $\alpha$ to proceed further. Indeed, $K_{14},K_{34},K_{35}$ are consecutive, so $K_{35}$ is out of the question. $K_{31}$ can be ruled out by Lemma~\ref{lem:conflict} for $A_3$. The remaining possibilities are $K_{32}$ and $K_{33}$.

{\em \ul{Case B.3.2.1.}} $\alpha$ continues through $K_{32}$. From there, it cannot continue through $K_{22}$ (apply Lemma~\ref{lem:generalisedconflict} to $K_{32}$). So it has to branch off via $K_{12}$ to $A_1$. From there, it cannot continue through $K_{15}$ since otherwise $\alpha$ would self intersect in $A_1$. $K_{11}$ is impossible as well, for $K_{32},K_{12},K_{11}$ are consecutive. $K_{15}$ can be ruled out using Lemma~\ref{lem:conflict} for $A_3$. So $\alpha$ can only continue through $K_{13}$, and from there only through $K_{23}$ ($K_{12},K_{13},K_{33}$ are consecutive). From $A_2$, it cannot go on through any band except $K_{25}$. Indeed, $K_{22}$ is impossible because $K_{13},K_{23},K_{22}$ are consecutive. $K_{21}$ and $K_{24}$ can be ruled out by applying Lemma~\ref{lem:generalisedconflict} to $(K_{34},K_{14})$ and $K_{23}$ respectively. After passing through $K_{25}$, $\alpha$ cannot go further: $K_{15}$ is impossible by Lemma~\ref{lem:generalisedconflict} (applied to $K_{23},K_{25}$) and $K_{35}$ can be ruled out by applying Lemma~\ref{lem:generalisedconflict} to $K_{34},K_{14},K_{11}$.

{\em \ul{Case B.3.2.2.}} $\alpha$ continues through $K_{33}$. Then, $K_{13}$ cannot be next since $K_{34},K_{32},K_{13}$ are consecutive. Thus $\alpha$ passes through $K_{23}$. From $A_2$, it cannot go on via $K_{24}$, for $K_{33},K_{23},K_{24}$ are consecutive. $K_{21}$ and $K_{22}$ are impossible as well (apply Lemma~\ref{lem:generalisedconflict} to $K_{14}$). So $\alpha$ has to go through $K_{25}$. Then, it cannot proceed through $K_{15}$ (apply Lemma~\ref{lem:generalisedconflict} to $K_{25}$). It cannot go via $K_{35}$ either (apply Lemma~\ref{lem:generalisedconflict} to $K_{14},K_{11}$), so $\alpha$ cannot continue at all.

{\em \ul{Case B.4.}} $\alpha$ continues through $K_{13}$. This is analogous to Case~B.3 and finishes the proof.
\end{proof}

\begin{proof}[Proof of Proposition \ref{prop:DnE7}]
We will present a case by case analysis for the possible clean arcs $\alpha$ in the fibre surface $S$ of each of $E_7$ and $D_n$. The reader interested in studying the proof is advised to follow the arguments along with a pencil and copies of Figures~\ref{fig:E7Hexagons} and~\ref{fig:DnSquares}, top and bottom. As in the proof of Proposition~\ref{prop:3345} above, we will make extensive use of Lemma~\ref{lem:consecutive} to exclude further polygon edges that $\alpha$ might cross on its way from its starting point to its end. In order to keep the proof short, we will usually refer to such situations by just saying "$\alpha$ is trapped", or by saying that an edge "is a trap", meaning that $\alpha$ would traverse too many consecutive bands to be clean.\\

{\bf ($E_7$)}\qquad First, let $S$ be the fibre surface of $E_7$, denote its monodromy $\phi$ and let $\alpha\subset S$ be a clean arc. Bring $\alpha$ into normal position with respect to $k_1,\ldots,k_9$. Note that the set of vertices of the hexagons $A_1,A_2,A_3$ decompose into two orbits under $\phi$, namely the orbit of the vertex of $A_1$ between $k_1$ and $k_2$, and the orbit of the vertex of $A_1$ between $k_2$ and $k_7$. We may therefore assume by Remark~\ref{rem:clean} that $\alpha$ starts at one of these two vertices.

{\em \ul{Case 1.}} $\alpha$ starts at the vertex of $A_1$ between $k_1$ and $k_2$. Define an involution $\tau:S\to S$ as follows: $\tau$ interchanges hexagons $A_1$ and $A_2$ and then reflects $A_1,A_2,A_3$ along the diagonals parallel to $k_7$, $k_8$, $k_1$ respectively, whereby it induces the permutation $(13)(49)(58)(67)$ on the edges $(k_1,\ldots,k_9)$. We have $\phi\circ\tau\circ\phi=\tau$, $\tau\circ\phi$ fixes the vertex of $A_1$ between $k_1$ and $k_2$ and swaps the edges $k_4,k_8$ as well as the edges $k_5,k_7$. By Remark~\ref{rem:clean}, we may therefore assume that $\alpha$ either continues through $k_4$ or through $k_5$.

\begin{figure}[h]
\includegraphics{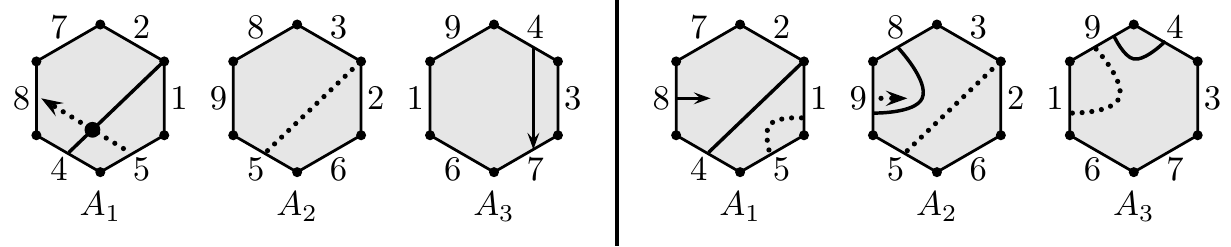}
\caption{Illustration of two of the steps in Case 1.1. The arc $\alpha$ is drawn as solid line, whereas $\phi(\alpha)$ is shown as a dotted line. \label{fig:E7Case11}}
\end{figure}

{\em \ul{Case 1.1.}} $\alpha$ continues through $k_4$. From $A_3$, it can only choose $k_9$. Indeed, $k_7$, $k_6$ and $k_1$ would imply an intersection in $A_1$ (Lemma~\ref{lem:conflict}, compare Figure~\ref{fig:E7Case11}, left), and $k_3$ is consecutive to $k_4$ (Lemma~\ref{lem:consecutive}, applied to the component with twist length one). Arriving in $A_2$, $k_2$ and $k_6$ would imply an intersection in $A_2$, so continuation is possible through $k_3$, $k_5$, $k_8$ only. But if $\alpha$ continues through $k_5$ or $k_8$, it will be trapped (compare Figure~\ref{fig:E7Case11}, right). Therefore it goes through $k_3$. Arriving in $A_3$, it has to go through $k_4$ ($k_9$ implies an intersection in $A_3$ and $k_1$, $k_6$, $k_7$ imply intersections in $A_1$). However, passing through $k_4$, $\alpha$ is trapped.

{\em \ul{Case 1.2.}} $\alpha$ continues through $k_5$. From $A_2$, it can continue through $k_3$, $k_6$, $k_8$ or $k_9$ ($k_2$ implies an intersection in $A_2$). If it passes through $k_6$ or $k_9$, it is trapped. So $k_8$ and $k_3$ are the only possiblities left.

{\em \ul{Case 1.2.1.}} $\alpha$ continues through $k_8$. Upon arrival in $A_1$, it cannot continue through $k_2$, $k_5$ (intersection in $A_2$) nor through $k_1$ (this would imply an intersection in $A_1$). But if it continues through either of $k_7$ or $k_4$, it is trapped.

{\em \ul{Case 1.2.2.}} $\alpha$ continues through $k_3$. From $A_3$, $\alpha$ cannot go on through $k_9$, $k_1$ (this would imply an intersection in $A_3$). If it passes through $k_4$, it is trapped. Suppose it continues through $k_6$. Arriving in $A_2$, it cannot continue through $k_9$, $k_8$, $k_3$ (this would produce an intersection in $A_3$), nor through $k_2$ (intersection in $A_2$). Finally, continuing through $k_5$, it will be trapped. Therefore $\alpha$ has to continue from $A_3$ through $k_7$. Arriving in $A_1$, it can continue through $k_2$, $k_8$ or $k_4$ ($k_1$ implies an intersection in $A_1$ and $k_5$ implies an intersection in $A_2$). But all of these are traps.

{\em \ul{Case 2.}} $\alpha$ starts at the vertex of $A_1$ between $k_2$ and $k_7$. Define an involution $\sigma:S\to S$ as follows: $\sigma$ interchanges $A_1$ and $A_2$ and then reflects $A_1,A_2,A_3$ along the diagonals parallel to $k_1$, $k_2$, $k_4$ respectively, inducing the permutation $(19)(28)(37)(46)$ on the edges. As in Case~1, we have $\phi\circ\sigma\circ\phi=\sigma$, and $\sigma\circ\phi$ fixes the vertex of $A_1$ between $k_2$ and $k_7$, swapping $k_4$ and $k_5$ as well as $k_1$ and $k_8$. By Remark~\ref{rem:clean}, we may therefore assume that $\alpha$ continues through either $k_1$ or $k_5$. However, if $\alpha$ continues through $k_1$, it is trapped. Therefore it continues through $k_5$. From $A_2$, it can continue through $k_8$ or $k_9$ ($k_6$ is a trap and $k_2$, $k_3$ imply intersections in $A_2$).

{\em \ul{Case 2.1.}} $\alpha$ continues through $k_8$. From $A_1$, it cannot continue through any of $k_2$, $k_1$, $k_5$, because this would produce an intersection in $A_2$, and $k_7$ is a trap. Therefore, it continues through $k_4$ and arrives in $A_3$. Continuation through $k_1$ produces an intersection in $A_3$, and $k_6$, $k_7$, $k_3$ imply intersections in $A_1$. Finally, $k_9$ is a trap.

{\em \ul{Case 2.2.}} $\alpha$ continues through $k_9$. Arriving in $A_3$, it can only continue through $k_1$ or $k_4$ (any other continuation produces an intersection in $A_1$). However, both $k_1$ and $k_4$ are traps, ending the proof for $E_7$.\\

{\bf ($D_n$, $n$ even)}\qquad Now, suppose $n$ is even and let $\alpha$ be a clean arc in the fibre surface $S$ of $D_n$ in normal position with respect to $k_1,\ldots,k_{n-1},\ k_1',\ldots,k_{n-1}'$. Define an involution $\tau:S\to S$ as follows: $\tau$ permutes the disks $A_i$ according to the rule $\tau(A_i)=A_{n-i+2}$ for $i=1,\ldots,n-1$ and then reflects every $A_i$ on the diagonal that contains the vertex between $k_i$ and $k_{i+2}$ (all indices are to be taken modulo $n$). Again $\phi\circ\tau\circ\phi=\tau$, and $\tau\circ\phi$ fixes the vertex of $A_1$ between $k_1'$ and $k_2'$ as well as the vertex between $k_1$ and $k_3$, and swaps the other two vertices. We may therefore assume that $\alpha$ starts at a vertex of $A_1$ which is not the vertex between $k_2'$ and $k_3$.

{\em \ul{Case 1.}} $\alpha$ starts at the vertex of $A_1$ between $k_1$ and $k_1'$. If it continues through $k_2'$, it is already trapped. So it has to continue through $k_3$. Arriving in $A_3$, it can continue through $k_3'$, $k_4'$ or $k_5$.

{\em \ul{Case 1.1.}} $\alpha$ continues from $A_3$ through $k_3'$. From $A_2$, it cannot continue through $k_2$ (otherwise it would intersect with $\phi(\alpha)$), so it can only proceed through $k_2'$ or $k_4$. However, both are traps.

{\em \ul{Case 1.2.}} $\alpha$ continues from $A_3$ through $k_4'$. This is similar to Case~1.1: arriving in $A_4$, $\alpha$ can only continue through $k_5'$ (which is a trap) or $k_4$. If it goes through $k_4$, it has to continue from $A_2$ through $k_3'$ ($k_2$ produces an intersection in $A_2$ and $k_2'$ produces an intersection in $A_3$). Then however, it is trapped again.

{\em \ul{Case 1.3.}} $\alpha$ can therefore continue from $A_3$ through $k_5$ only. In $A_5$, the same situation reproduces, except that all indices in consideration are now shifted by $+2$. Therefore the only way for $\alpha$ to continue from $A_5$ is by passing through the edges $k_7,k_9,k_{11},\ldots$ After at most $n/2$ more steps, $\alpha$ will be trapped.

{\em \ul{Case 2.}} $\alpha$ starts at the vertex of $A_1$ between $k_1'$ and $k_2'$. Using $\tau$ again, we may assume that it continues through $k_1$ to $A_{n-2}$. If it goes through $k_{n-1}'$ next, it is trapped since it is forced to follow the sequence of edges $k_{n-1},k_{n-2}',k_{n-2},k_{n-3}',\ldots$ If it goes through $k_{n-2}'$ to $A_{n-3}$ instead, it can only continue from there through $k_{n-3}'$ or $k_{n-1}$, and these are traps again. So it has to continue from $A_{n-2}$ through $k_{n-2}$. In $A_{n-4}$, the same situation as one step earlier (where $\alpha$ arrived through $k_1$ in $A_{n-2}$) reproduces, except that all indices appearing in the consideration are now shifted by $-2$. Hence the only way $\alpha$ can continue from $A_{n-4}$ is by going through the sequence of edges $k_{n-4},k_{n-6},k_{n-8},\ldots$ After at most $n/2$ steps, $\alpha$ will be trapped.

{\em \ul{Case 3.}} $\alpha$ starts at the vertex of $A_1$ between $k_1$ and $k_3$. Using the involution $\tau$ from above, we may assume that it continues through $k_2'$. From $A_2$, it cannot go on through $k_4$, for this would imply an intersection in $A_2$. However, the two possibilities that remain ($k_3'$ and $k_2$) are traps, which ends the proof for $D_n$, $n$ even.\\

{\bf ($D_n$, $n$ odd)}\qquad Finally, let $n$ be odd and let $S$ be the fibre surface of $D_n$. Suppose again we have a clean arc $\alpha\subset S$ in normal position with respect to $k_1,\ldots,k_{2n-2}$. Since the monodromy permutes the $A_i$ cyclically and since there are only two orbits of vertices of the $A_i$, we may assume that $\alpha$ starts in $A_1$, at the vertex between $k_1$ and $k_2$, or at the vertex between $k_2$ and $k_n$. As before, we then make use of Remark~\ref{rem:clean} with the help of the involution $\tau:S\to S$ defined as follows: $\tau(A_i)=A_{n-i+2}$ by translations followed by a reflection on the diagonal of $A_i$ that contains the vertex between $k_{n+i-1}$ and $k_{n+i}$ for $i=1,2$ and reflection on the diagonal of $A_i$ that contains the vertex between $k_i$ and $k_{i+1}$ for $i=3,\ldots,n-1$. Applying Remark~\ref{rem:clean} as before, we may assume that $\alpha$ either starts at the vertex of $A_1$ between $k_1,k_2$ and continues through $k_n$ (say), or that it starts at the vertex of $A_1$ between $k_2$ and $k_n$, continuing through $k_1$ (say). So there are two cases to consider, one being very similar to Case~1 above and the other similar to Case~3. No new arguments are needed.
\end{proof}

\end{document}